\documentclass[reqno,11pt]{amsart}
\usepackage[letterpaper, margin=1in]{geometry}
\RequirePackage{amsmath,amssymb,amsthm,graphicx,mathrsfs,url,slashed,subcaption,empheq}
\RequirePackage[usenames,dvipsnames]{xcolor}
\RequirePackage[colorlinks=true,linkcolor=Red,citecolor=Green]{hyperref}
\RequirePackage{amsxtra}
\usepackage{cancel}
\usepackage{tikz, quiver, wrapfig, enumitem}

\makeatletter
\@namedef{subjclassname@2020}{\textup{2020} Mathematics Subject Classification}
\makeatother

\title{The max flow/min cut theorem for currents and laminations}
\author{Aidan Backus}
\address{Department of Mathematics, Brown University}
\email{aidan\_backus@brown.edu}
\date{\today}
\thanks{}

\newcommand{\QQ}{\mathbf{Q}}
\newcommand{\RR}{\mathbf{R}}
\newcommand{\CC}{\mathbf{C}}

\newcommand{\PP}{\mathbf P}

\newcommand{\Sph}{\mathbf S}

\newcommand*\dif{\mathop{}\!\mathrm{d}}

\DeclareMathOperator{\area}{area}
\DeclareMathOperator{\card}{card}

\DeclareMathOperator{\Div}{div}

\DeclareMathOperator{\flux}{flux}

\DeclareMathOperator*{\esssup}{ess\,sup}

\DeclareMathOperator{\Hom}{Hom}
\DeclareMathOperator{\mass}{mass}

\DeclareMathOperator{\supp}{supp}

\newcommand{\weakto}{\rightharpoonup}
\newcommand{\restr}{\upharpoonright}

\newcommand{\normal}{\mathbf n}

\newcommand{\vol}{\mathrm{vol}}

\newcommand{\len}{\mathrm{len}}

\newcommand{\dfn}[1]{\emph{#1}\index{#1}}

\newcommand{\cl}{\mathrm{cl}}

\newcommand{\loc}{\mathrm{loc}}
\newcommand{\cpt}{\mathrm{cpt}}

\newtheorem{theorem}{Theorem}[section]

\newtheorem{lemma}[theorem]{Lemma}

\newtheorem{corollary}[theorem]{Corollary}

\theoremstyle{definition}
\newtheorem{definition}[theorem]{Definition}

\newtheorem{construction}[theorem]{Construction}



\newlist{prfenum}{enumerate}{1}
\setlist[prfenum]{
  nosep,
  label=(\theprfenum--\arabic*),
}
\newcounter{prfenum}
\AddToHook{env/proof/begin}{\setcounter{prfenum}{0}}
\AddToHook{env/construction/begin}{\setcounter{prfenum}{0}}
\AddToHook{env/prfenum/begin}{\stepcounter{prfenum}}

\numberwithin{equation}{section}

\def\XXint#1#2#3{{\setbox0=\hbox{$#1{#2#3}{\int}$ }
\vcenter{\hbox{$#2#3$ }}\kern-.6\wd0}}

\usepackage[backend=bibtex,style=alphabetic,giveninits=true]{biblatex}

\renewbibmacro{in:}{}
\DeclareFieldFormat{pages}{#1}

\begin{document}
\begin{abstract}
Motivated by applications to holography and Teichm\"uller theory, we prove a continuous analogue of the max flow/min cut theorem which also takes the topology of the domain into account.
\end{abstract}

\maketitle

%%%%%%%%%%%%%%%%%%%%%%%%%%%%%%%%%%%%%%%%%%%%%%%%%%%%%%%

\section{Introduction}
A fundamental theorem in combinatorial optimization is the \dfn{max flow/min cut theorem}, which asserts that the maximal possible amount of flow through a network is equal to the capacity of the smallest bottleneck.
More precisely, a \dfn{flow network} $(V, E, s_0, s_1)$ is a finite directed graph $(V, E)$, equipped with a \dfn{source} $s_0 \in V$ and a \dfn{sink} $s_1 \in V$, such that there is a path from $s_0$ to $s_1$.
A \dfn{flow} is a function $F: E \to [-1, 1]$ which satisfies the \dfn{conservation law} that for every $v \in V \setminus \{s_0, s_1\}$, the total flow into $v$ equals the total flow out of $v$:
$$\sum_{(u, v) \in E} F(u, v) = \sum_{(v, w) \in E} F(v, w).$$
A \dfn{cut} is a partition $V = S_0 \sqcup S_1$ with $s_i \in S_i$; we identify a cut with the set $C \subseteq E$ of edges from $S_0$ to $S_1$.
Thus in the graph $(V, E \setminus C)$, there is no path from $s_0$ to $s_1$.
It follows from the conservation law that 
\begin{equation}\label{eqn: discrete one sided MFMC}
\sum_{(s_0, v) \in E} F(s_0, v) \leq \min_{C^*} \card C^*
\end{equation}
where the minimum ranges over cuts and $\card C$ is the cardinality of the set $C$.
The max flow/min cut theorem asserts that this inequality is sharp:

\begin{theorem}[{\cite[Theorem 6.2]{papadimitriou1982combinatorial}}]\label{thm: MFMC discrete}
In every flow network $(V, E, s_0, s_1)$,
$$\max_{F^*} \sum_{(s_0, v) \in E} F^*(s_0, v) = \min_{C^*} \card C^*,$$
where the maximum ranges over flows and the minimum ranges over cuts.
\end{theorem}

The max flow/min cut theorem is useful enough that in the literature, one often sees allusions to a continuous analogue of it. 
Such allusions appear in such diverse fields as computational geometry \cite{sullivan1990crystalline}, Teichm\"uller theory \cite{Thurston98}, and holography \cite{Freedman_2016}.
A continuous theorem is available on bounded Lipschitz domains in euclidean space \cite{Strang1983, Nozawa90}, and in two dimensions this is essentially equivalent to Kantorovich duality for optimal transport \cite{Dweik2019}.
There is also a variant of the max flow/min cut theorem for $L^1$ Rayleigh quotients with applications to spectral geometry \cite{Grieser05}.
These theorems all neglect the role of the topology of the domain, but in \cite{Thurston98,Freedman_2016}, the topology has an essential role.
We therefore seek to formulate a continuous max flow/min cut theorem which takes the topology of the domain into account.
We refer to \cite[Figure 5]{Freedman_2016} for an elegant depiction of what we have in mind.

We shall formulate the continuous max flow/min cut theorem on a Riemannian manifold $M$ of dimension $d \geq 1$.
The manifold shall be compact (corresponding to finiteness), connected (corresponding to a path from source to sink), oriented (corresponding to directedness), and with boundary $\partial M$ (corresponding to the source and sink).
We write $\overline M := M \cup \partial M$.

Let $\alpha \in H_{d - 1}((M, \partial M), \RR)$, and let $S \subset \partial M$ be a $d - 2$-cycle with real coefficients such that $[S]$ is the image of $\alpha$ under the connecting homomorphism 
% https://q.uiver.app/#q=WzAsMixbMCwwLCJIX3tkIC0gMX0oKE0sIFxccGFydGlhbCBNKSwgXFxSUikiXSxbMSwwLCJIX3tkIC0gMn0oXFxwYXJ0aWFsIE0sIFxcUlIpIl0sWzAsMSwiXFxwYXJ0aWFsIl1d
\begin{equation}\label{eqn: homology compatibility}
\begin{tikzcd}
	{H_{d - 1}((M, \partial M), \RR)} & {H_{d - 2}(\partial M, \RR)}
	\arrow["\partial", from=1-1, to=1-2].
\end{tikzcd}
\end{equation}
Tentatively we think of cuts as corresponding to $d - 1$-chains $C \subset \overline M$ with real coefficients such that $[C] = \alpha$ and $\partial C = S$, and such that the Hausdorff measure, $\mathcal H^{d - 1}(C \cap \partial M)$, is zero.
Roughly speaking, these conditions imply that $C$ divides $\partial M$ into two regions, a ``source'' and a ``sink.''

A flow will be a vector field $F$ with $\|F\|_{L^\infty} \leq 1$, which satisfies the ``conservation law'' that $\Div F = 0$.
Since $[C] = \alpha$ and $\partial S = C$, and we're working with \emph{real} homology, $\int_C \flux F$ is independent of $C$.\footnote{Equivalently, one can think of $F$ as a \dfn{calibration} of codimension $1$, see \cite{Harvey82}.}
Analogously to (\ref{eqn: discrete one sided MFMC}), if there exists a cut $C^*$ which minimizes its area among all cuts, then
\begin{equation}\label{eqn: one sided MFMC}
\int_C \flux F \leq \area(C^*).
\end{equation}
Motivated by Theorem \ref{thm: MFMC discrete}, we shall show in Theorem \ref{thm: MFMC} that (\ref{eqn: one sided MFMC}) is sharp.

%%%%%%%%%%%%%%%%%%
\subsection{Some motivating examples}
\begin{figure}[t]
\includegraphics[width=0.35\textwidth]{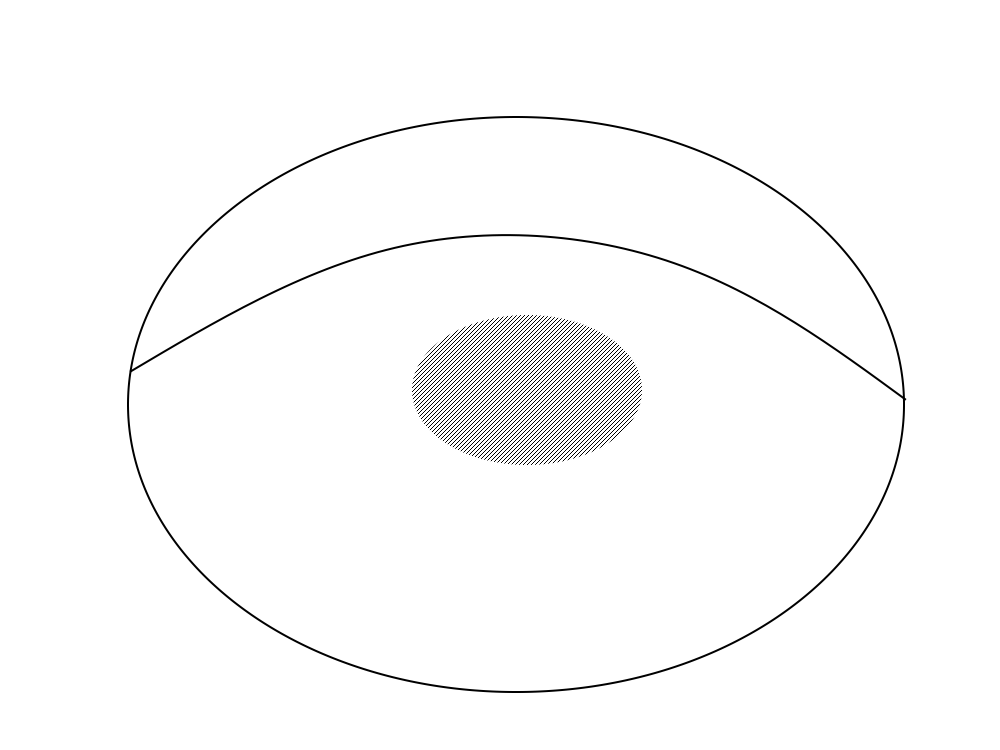} 
\includegraphics[width=0.35\textwidth]{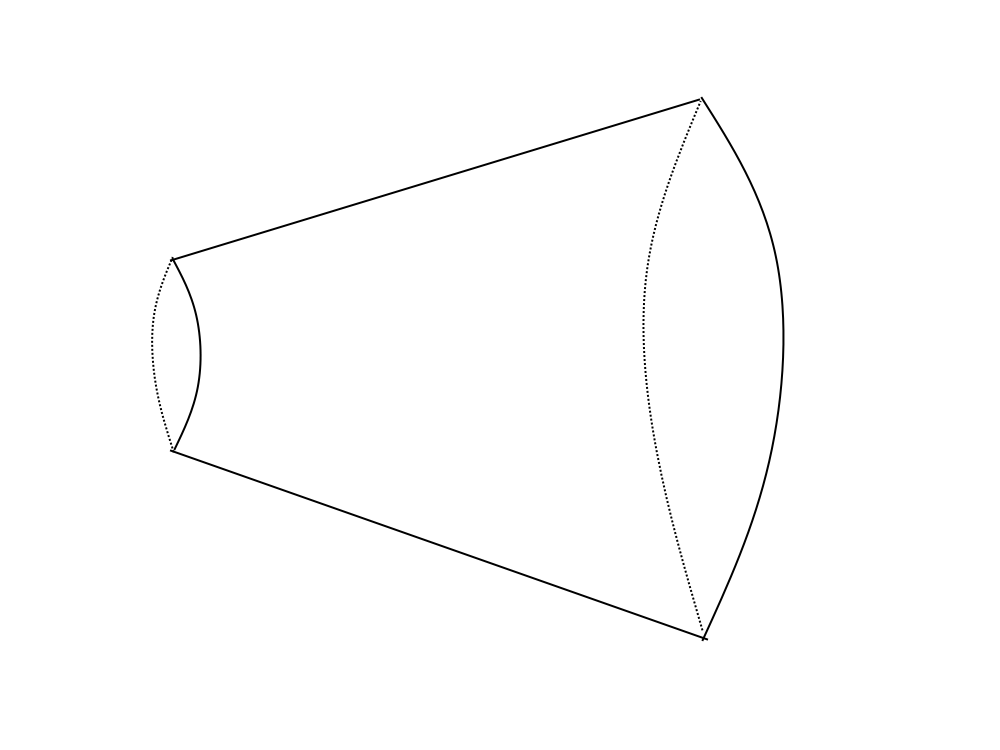}
\centering
\caption{Left: We have to specify a relative homology class to determine how the minimal cut will avoid the ``black hole region'' $\{x \leq 0\}$, shaded in grey. Right: Because the boundary includes a minimizing cycle, there is no minimal cut. \label{fig: basic examples}}
\end{figure}

To explain what we mean by ``taking the topology of the domain into account,'' suppose that $M = (-1, 1)_x \times \Sph^1_\theta$ with the metric $\dif x^2 + \cosh^2 x \dif \theta^2$, let $S = p - q$, where $p, q \in \partial M$.
Then any geodesic $C^*$ such that $\partial C^* = S$ must remain in the region $\{0 < x < 1\}$, but there are two such geodesics of that form, corresponding to $\pm \alpha$ where $\alpha$ is a unit generator of $H_1((M, \partial M), \RR)$, see Figure \ref{fig: basic examples}, left.
So it is not enough to split $\{x = 1\}$ along $S$ into a source and a sink; one actually has to choose the class $\alpha \in H_1((M, \partial M), \RR)$ of a cut, and in order that a cut actually exist, we must have the compatibility condition $[S] = \partial \alpha$.
The class $\alpha$ is natural in view of the applications: in holography, barriers similar to $\{x = 0\}$ can be interpreted as event horizons, and in Teichm\"uller theory, one typically has $\partial M = \emptyset$, so $C^*$ is completely determined by $\alpha$.

The max flow/min cut theorem fails in the presence of torsion.
Indeed, if $M = \RR\PP^2$ then $M$ admits exactly one minimizing geodesic $C^*$; if there was a corresponding maximal flow $F^*$ then the lift of $C^*$ to $\Sph^2$ would still be minimizing, a contradiction.
So we must work with coefficients in a field $K$, and the only reasonable options are $K = \RR$ and $K = \QQ$.
In Teichm\"uller theory one is interested in measured geodesic laminations without closed leaves, whose homology classes are irrational; meanwhile, in holography, restricting to rational coefficients seems rather unphysical.
So we are mainly interested in the case that $K = \RR$, but then the minimal cut need not actually be a chain. 
Indeed, if $M$ is the flat square torus and the homology class $\alpha$ has irrational direction, then $C^*$ is actually an irrational foliation.
To avoid this obstruction, we allow cuts to be $d - 1$-currents of finite mass.

The mean curvature of $\partial M$ must enter the picture.
Indeed, suppose that $M = (0, 2)_x \times \Sph^1_\theta$ with the metric $\dif x^2 + \cosh^2 x \dif \theta^2$, let $S$ be trivial, and let $\alpha$ be a unit generator of $H_1((M, \partial M), \RR)$, as in Figure \ref{fig: basic examples}, right.
Then $C^* = \{x = 0\}$, even though we sought to impose $\mathcal H^1(C \cap \partial M) = 0$.
To rule this out, we require that $\partial M$ be strictly mean convex.
In the application to holography, $M$ will typically be a large ball in an asymptotically hyperbolic manifold, and therefore will be strictly mean convex.

Finally, since we have been forced into the category of currents, we must be careful about the boundary regularity.
Suppose that $M$ is the unit disk in $\RR^2$.
Then there is a fat Cantor set in $\partial M$, whose boundary is a $0$-current $S$ such that there is no mass-minimizing current $C$ in $M$ such that $\partial C = S$ \cite{Spradlin14}.
However, $S$ does not have finite mass, so we avoid this issue by requiring that $S$ have finite mass.

%%%%%%%%%%%%%%
\subsection{Statement of the main theorem}
Let $\mathcal F$ be the set of all measurable vector fields $F$ on $M$ such that $\|F\|_{L^\infty} \leq 1$ and $\Div F = 0$.
See \S\ref{sec: Lefschetz section} for a review of currents on manifolds-with-boundary.

\begin{theorem}\label{thm: MFMC}
Suppose that $\partial M$ is strictly mean convex.
Let $\alpha \in H_{d - 1}((M, \partial M), \RR)$ and let $S$ be a closed $d - 2$-current of finite mass in $\partial M$ such that $[S] = \partial \alpha$.
Then there exists a $d - 1$-current $C^*$ which minimizes its mass among all currents $C$ such that $\partial C = S$ and $[C] = \alpha$.
Moreover,
\begin{equation}\label{eqn: MFMC}
\mass(C^*) = \max_{F^* \in \mathcal F} \int_{C^*} \flux F^*.
\end{equation}
\end{theorem}

The existence of the maximizing vector field $F^*$ is an easy argument based on the Hanh-Banach theorem \cite[Theorem 4.10(1)]{Federer1974}.
We include a separate proof based on the $p$-Laplacian, which in view of results of \cite{daskalopoulos2022,daskalopoulos2023} is more likely to be useful to prove generalizations of Theorem \ref{thm: Thurston}; we hope to return to this point in later work.

The existence of the minimizing current $C^*$ is not so easy, because the classical literature on Plateau's problem has been concerned with \emph{integral} currents.
A primitive of a minimizing (real) current is called a \dfn{function of least gradient}, and they have received much attention from analysts in recent years.
A modification of the arguments of \cite{Jerrard18, Górny2024} shows that the Dirichlet problem for functions of least gradient can be solved and so, after some work to set up the formalism of currents, we construct $C^*$.

On the other hand, when we restrict to rational chains, the situation greatly simplifies.
In that case, we can rescale so that all chains have integer multiplicity, so by the solution of Plateau's problem (for example, by easy generalizations of \cite[Chapter 7, Lemma 2.1 and Theorem 5.8]{simon1983GMT}), we obtain the following special case, which was already known.

\begin{corollary}\label{crly: rational MFMC}
In the situation of Theorem \ref{thm: MFMC}, if $\alpha \in H_{d - 1}((M, \partial M), \QQ)$ and $S$ is a rational $d - 2$-cycle, $C^*$ can be chosen to be a rational $d - 1$-chain of disjoint area-minimizing hypersurfaces, smooth away from a set of Hausdorff codimension $8$.
\end{corollary}

We have already seen that without rationality assumptions, $C^*$ need not be a chain.
However, by combining Theorem \ref{thm: MFMC} with \cite[Theorem B]{BackusCML}, it immediately follows that $C^*$ is a measured oriented lamination.

\begin{definition}
A (codimension-$1$) \dfn{measured oriented lamination} $\lambda$ consists of:
\begin{enumerate}
\item A set of disjoint smooth connected complete hypersurfaces, called \dfn{leaves}, whose union is a closed set, the \dfn{support} of $\lambda$.
\item An oriented Lipschitz atlas of coordinate charts $U \cong (0, 1)^d$, and closed sets $K_U \subseteq (0, 1)$, in which $\supp \lambda \cap U \cong K_U \times (0, 1)^{d - 1}$, such that for each leaf $N$ and each chart $U$ there exists $k_U \in K_U$ such that $N \cap U \cong \{k_U\} \times (0, 1)^{d - 1}$.
\item A relatively closed $d - 1$-current of finite mass, which we also call $\lambda$, such that for each coordinate chart $U$ there is a finite Borel measure $\mu_U$ with $\supp \mu_U = K_U$, such that for every $\psi \in C^0_\cpt(U, \Omega^{d - 1})$,
$$\int_\lambda \psi = \int_{K_U} \left[\int_{\{k\} \times (0, 1)^{d - 1}} \psi\right] \dif \mu_U(k).$$
\end{enumerate}
The lamination $\lambda$ is \dfn{minimal} if all of its leaves have zero mean curvature, and \dfn{geodesic} if, in addition, $d = 2$.
\end{definition}

\begin{corollary}\label{crly: lamination MFMC}
In the situation of Theorem \ref{thm: MFMC}, if $d \leq 7$, then $C^*$ is a measured oriented minimal lamination.
\end{corollary}

In \S\ref{sec: MFMC proof}, we prove Theorem \ref{thm: MFMC}.
In \S\ref{sec: epilogue} we state a sharper, but more technical, version of Theorem \ref{thm: MFMC}; we also give a few conjectures and make a comment on the connection to Teichm\"uller theory.

%%%%%%%%%%%%%%%%%%%%%%%%%%%%%
\subsection{Acknowledgments}
I would like to thank Yi Huang for asking me how to interpret Theorem \ref{thm: Thurston} as a max flow/min cut theorem, which was the question which motivated this work. I would also like to thank Georgios Daskalopoulos and Stefan Steinerberger for valuable comments on a previous draft. This research was supported by the National Science Foundation's Graduate Research Fellowship Program under Grant No. DGE-2040433.

%%%%%%%%%%%%%%%%%%%%%%%%%%%%%%%%%%%
\section{Currents on manifolds-with-boundary}\label{sec: Lefschetz section}
We establish some notation.
Let $\overline M = M \cup \partial M$ be a compact Riemannian manifold-with-boundary.
Let $\Omega^\bullet$ be the exterior algebra over the cotangent bundle of $M$, $\Omega^\bullet_\cl$ be the subsheaf of closed forms, $\star$ be the Hodge star, and $\mathcal H^s$ be the $s$-dimensional Hausdorff measure.
Let $\iota: \partial M \to \overline M$ be the inclusion map, which induces a pullback map $\iota^*$ on sections of $\Omega^\bullet$.
The \dfn{flux} of a vector field $F$ is the $d - 1$-form such that for any oriented hypersurface $L \subset M$ with unit normal $\normal$,
$$\int_L \flux F = \int_L \langle \normal, F\rangle \dif \mathcal H^{d - 1}.$$

A \dfn{$k$-current} in $M$ is a continuous linear function on $C^\infty_\cpt(M, \Omega^k)$.
Given a $k$-current $T$ and $\varphi \in C^\infty_\cpt(M, \Omega^k)$, $\int_T \varphi$ means $T(\varphi)$.
The $k$-current $T$ has \dfn{finite mass} if there exists $C > 0$ such that for every $\varphi \in C^\infty_\cpt(M, \Omega^k)$,
\begin{equation}\label{eqn: mass estimate}
\left|\int_T \varphi\right| \leq C \|\varphi\|_{L^\infty}
\end{equation}
and the smallest constant $C$ satisfying (\ref{eqn: mass estimate}) is called $\mass(T)$.\footnote{Since we are mainly interested in currents of codimension $1$, the distinction between the comass of $\varphi$ \cite[Chapter 6, Remark 2.7]{simon1983GMT} and $\|\varphi\|_{L^\infty}$ turns out to be immaterial, so we suppress it.}
By a routine cutoff argument, one can show that every $k$-current of finite mass extends to $C^0(\overline M, \Omega^k)$; Anzellotti's theorem gives a further extension.

% \begin{construction}
% Let $T$ be a $k$-current of finite mass in $M$ and $\varphi \in C^0(\overline M, \Omega^k)$.
% We define $\int_T \varphi$.
% Choose a compact exhaustion $(K_n)$ of $M$ and a \dfn{cutoff} $(\chi_n)$ subordinate to $(K_n)$, that is, an increasing sequence in $C^\infty_\cpt(M, [0, 1])$ such that $\chi_n \restr K_n = 1$.
% Since $T$ has finite mass, it extends continuously to $C^0_\cpt(\overline M, \Omega^k)$, so we can set
% $$\int_T \varphi := \lim_{n \to \infty} \int_T \chi_n \varphi.$$
% That this is well-defined follows from the fact that if $\chi^\heartsuit, \chi^\diamondsuit \in C^\infty_\cpt(M, [0, 1])$ and $(\chi^\heartsuit - \chi^\diamondsuit) \restr K_n = 0$, then 
% $$\left|\int_T \chi^\heartsuit \varphi - \int_T \chi^\diamondsuit \varphi\right| \lesssim \|\varphi\|_{L^\infty} \mass(T \restr (M \setminus K_n)).$$
% Since $T$ has finite mass, $\mass(T \restr \partial M) = 0$, so by continuity of measure from above, $\mass(T \restr (M \setminus K_n)) \to 0$ as $n \to \infty$.
% In particular, $(\int_T \chi_n \varphi)$ is a Cauchy sequence whose limit does not depend on the choice of cutoff (and hence not on the choice of compact exhaustion).
% Since $\|\varphi\|_{L^\infty}$ is the increasing limit of $\|\chi_n \varphi\|_{L^\infty}$, $\varphi$ satisfies the mass bound (\ref{eqn: mass estimate}).
% \end{construction}

\begin{theorem}[Anzellotti's theorem, {\cite{Anzellotti1983}, \cite[Theorem 2.5]{BackusBest2}}]\label{thm: Anzellotti}
Every relatively closed $k$-current of finite mass in $M$ extends to the space of all $\varphi \in L^\infty(M, \Omega^k)$ such that $\dif \varphi \in L^d(M, \Omega^k)$, and for any such $\varphi$, the mass bound (\ref{eqn: mass estimate}) holds.
\end{theorem}

Given an open set $U \subseteq \overline M$ with Lipschitz boundary, $T \restr U$ means the restriction of $T$ to $C^\infty_\cpt(U, \Omega^k)$.
The current $T$ is \dfn{relatively closed} if $T$ annihilates the image of 
% https://q.uiver.app/#q=WzAsMixbMCwwLCJDXlxcaW5mdHlfXFxjcHQoTSwgXFxPbWVnYV57ayAtIDF9KSJdLFsxLDAsIkNeXFxpbmZ0eV9cXGNwdChNLCBcXE9tZWdhXmspIl0sWzAsMSwiXFxkaWYiXV0=
\[\begin{tikzcd}
	{C^\infty_\cpt(M, \Omega^{k - 1})} & {C^\infty_\cpt(M, \Omega^k)}
	\arrow["\dif", from=1-1, to=1-2]
\end{tikzcd}\]
Since a current $T$ in $M$ is a continuous linear function on $C^\infty_\cpt(M, \Omega^k)$, in the classical theory of currents, \cite{simon1983GMT}, the boundary $\partial T$ does not detect the behavior of $T$ near $\partial M$.
As intimated by the fact that the pullback of a distribution is only defined along a submersion \cite[Chapter 6]{hörmander2007analysis}, one cannot in general define $\partial T$ to include information about the germ of $T$ along $\partial M$.
We here extend the boundary map to currents which have finite mass.
The formulation here was suggested by an anonymous MathOverflow user.

\begin{construction}\label{cons: boundary homomorphism}
We define the boundary homomorphism
\begin{equation}\label{eqn: current in direct sum}
% https://q.uiver.app/#q=WzAsMixbMCwwLCJDXjBfXFxjcHQoTSwgXFxPbWVnYV5rKSJdLFsxLDAsIkNeMV9cXGNwdChNLCBcXE9tZWdhXntrIC0gMX0pXiogXFxvcGx1cyBDXjFfXFxjcHQoXFxwYXJ0aWFsIE0sIFxcT21lZ2Fee2sgLSAxfSleKiJdLFswLDEsIlxccGFydGlhbCJdXQ==
\begin{tikzcd}
	{C^0(\overline M, \Omega^k)^*} & {C^1(\overline M, \Omega^{k - 1})^* \oplus C^1(\partial M, \Omega^{k - 1})^*}
	\arrow["\partial", from=1-1, to=1-2]
\end{tikzcd}
\end{equation}
Let $T$ be a $k$-current of finite mass, and for each $\eta \in C^1_\cpt(M, \Omega^{k - 1})$, let $\int_{S_1} \eta := \int_T \dif \eta$.
Then $S_1$ extends to a current in $C^1(\overline M, \Omega^{k - 1})^*$, by a cutoff argument.
For each $\psi \in C^1(\partial M, \Omega^{k - 1})$, choose $\overline \psi \in C^1(\overline M, \Omega^{k - 1})$ such that $\iota^* \overline \psi = \psi$ and let
$$\int_{S_2} \psi := \int_T \dif \overline \psi - \int_{S_1} \overline \psi.$$
Then $\partial T := (S_1, S_2)$.

To show that $\partial$ is well-defined, we must show that if $\psi = 0$ then
\begin{equation}\label{eqn: boundary is welldefined}
\int_T \dif \overline \psi = \int_{S_1} \overline \psi.
\end{equation}
We may assume that $M = \RR^d_+$, $\partial M = \RR^{d - 1}$, and a typical point is $(x, y)$ where $x \in \RR^{d - 1}$ and $y > 0$.
Choose $K_n := \{y > 1/n\}$ and $\chi_n \in C^\infty_\cpt(K_n, [0, 1])$ such that $\chi_n \restr K_{n - 1} = 1$, $\|\dif \chi_n\|_{L^\infty} \lesssim n$, and $\partial_x \chi_n = 0$.
Therefore, if $P$ denotes the orthogonal projection which annihilates $\dif y$,
$$\dif \chi_n \wedge \overline \psi = \dif \chi_n \wedge P\overline \psi.$$
But since 
$$P\overline \psi(x, 0) = \iota^* \overline \psi(x, 0) = 0,$$
and $\overline \psi \in C^1(\overline M, \Omega^{k - 1})$, we obtain by Taylor expansion that 
$$|P\overline \psi(x, y)| \lesssim |y|$$
and in particular that
$$\|P\overline \psi\|_{L^\infty(M \setminus K_n)} \lesssim \frac{1}{n}.$$
By definition of $S_1$, and since $T$ has finite mass, 
\begin{align*}
\left|\int_T \dif \overline \psi - \int_{S_1} \overline \psi\right| 
&= \lim_{n \to \infty} \left|\int_T \chi_n \dif \overline \psi - \int_{S_1} \chi_n \overline \psi\right| \\
&= \lim_{n \to \infty} \left|\int_T \dif \chi_n \wedge \overline \psi\right| \\
&\lesssim \lim_{n \to \infty} n \mass(T \restr (M \setminus K_n)) \cdot \|P\overline \psi\|_{L^\infty(M \setminus K_n)} \\
&= 0
\end{align*}
which completes the proof of (\ref{eqn: boundary is welldefined}).
\end{construction}

Since $H^\bullet((M, \partial M), \RR)$ is the cohomology of the chain complex $C^\infty_\cpt(M, \Omega^\bullet)$, every relatively closed $k$-current $T$ has a homology class $[T] \in H_k((M, \partial M), \RR)$.
It follows from the zigzag lemma and de Rham's theorem that the connecting homomorphism in homology (\ref{eqn: homology compatibility}) is induced by the boundary operator (\ref{eqn: current in direct sum}).

Let $K$ be a subfield of $\CC$.
A routine argument shows that the map
\begin{align*}
C^0(\overline M, \Omega^k_\cl) &\to C^0_\cpt(M, \Omega^{d - k})^* \\
\eta &\mapsto \left(\varphi \mapsto \int_M \eta \wedge \varphi\right)
\end{align*}
induces isomorphisms
\begin{equation}\label{eqn: Lefschetz isomorphism}
H^k(M, K) \to H_{d - k}((M, \partial M), K).
\end{equation}
For each $\alpha \in H^{d - k}((M, \partial M), K)$, the \dfn{Lefschetz dual} of $\alpha$, $\hat \alpha \in H^k(M, K)$, is the preimage of $\alpha$ under (\ref{eqn: Lefschetz isomorphism}).

\begin{lemma}\label{lma: Lefschetz duality commutes with pullback}
The diagram 
% https://q.uiver.app/#q=WzAsNCxbMCwwLCJIX3tkIC0ga30oKE0sIFxccGFydGlhbCBNKSwgXFxSUikiXSxbMCwxLCJIXmsoTSxcXFJSKSJdLFsyLDAsIkhfe2QgLSBrIC0gMX0oXFxwYXJ0aWFsIE0sIFxcUlIpIl0sWzIsMSwiSF5rKFxccGFydGlhbCBNLCBcXFJSKSJdLFswLDFdLFswLDIsIigtMSleayBcXHBhcnRpYWwiXSxbMSwzLCJcXGlvdGFeKiJdLFsyLDNdXQ==
\[\begin{tikzcd}
	{H_{d - k}((M, \partial M), \RR)} && {H_{d - k - 1}(\partial M, \RR)} \\
	{H^k(M,\RR)} && {H^k(\partial M, \RR)}
	\arrow["{(-1)^k \partial}", from=1-1, to=1-3]
	\arrow[from=1-1, to=2-1]
	\arrow[from=1-3, to=2-3]
	\arrow["{\iota^*}", from=2-1, to=2-3]
\end{tikzcd}\]
commutes, where the vertical arrows are Lefschetz duality and $\partial$ is the connecting homomorphism of (\ref{eqn: homology compatibility}).
\end{lemma}
\begin{proof}
Let $\alpha \in H_{d - k}((M, \partial M), \RR)$ and choose $\eta \in C^\infty(\overline M, \Omega^k_\cl)$ such that $[\eta] = \hat \alpha$.
Define for each $\varphi \in C^\infty(\overline M, \Omega^{d - k})$,
$$\int_{T_\eta} \varphi := \int_M \eta \wedge \varphi.$$
Then for every $\psi \in C^\infty(\overline M, \Omega^{d - k - 1})$,
$$\int_{\partial T_\eta} \psi = \int_{T_\eta} \dif \psi = \int_M \eta \wedge \dif \psi = (-1)^k \int_{\partial M} \eta \wedge \psi = (-1)^k \int_{\partial M} \iota^* \eta \wedge \psi.$$
Therefore $\iota^* \hat \alpha$ is the Lefschetz dual of
\begin{align*}
(-1)^k [\partial T_\eta] &= (-1)^k \partial [T_\eta] = (-1)^k \partial \alpha. \qedhere
\end{align*}
\end{proof}

Let $u$ be a distribution.
If the $d - 1$-current $\varphi \mapsto \int_M \dif u \wedge \varphi$ has finite mass, then we write $u \in BV(M)$ and denote its mass by $\int_M |\dif u| \dif V$.
If $u \in W^{1, 1}(M)$, then this quantity agrees with the usual $L^1$ norm of $\dif u$.

\begin{theorem}[$BV$ trace theorem, {\cite[Chapter 2]{Giusti77}}]\label{thm: trace}
For every $u \in BV(M)$, there exists a unique $u \restr \partial M \in L^1(\partial M)$ such that for every $\varphi \in C^1(\overline M, \Omega^{d - 1})$,
$$\int_M \dif u \wedge \varphi + \int_M u \dif \varphi = \int_{\partial M} (u \restr \partial M) \varphi.$$
\end{theorem}

Let $u \in BV(M)$.
When no confusion can occur, we write $u$ to mean its trace $u \restr \partial M$.
It follows from the $BV$ trace theorem and Construction \ref{cons: boundary homomorphism} that the current $\int_T \varphi := \int_M \dif u \wedge \varphi$ is relatively closed, and for every $\psi \in C^1(\partial M, \Omega^{d - 2})$,
\begin{equation}\label{eqn: boundary current is trace}
\int_{\partial T} \psi = \int_{\partial M} u \dif \psi.
\end{equation}

In the $BV$ calculus of variations, one needs a slightly stronger trace notion \cite{Jerrard18}.
A function $f \in L^1(\partial M)$ is the \dfn{$L^\infty$ trace} of $u \in BV(M)$, if for $\mathcal H^{d - 1}$-almost every $x \in \partial M$, and every $\varepsilon > 0$, there exists an open neighborhood $U$ of $x$ such that
$$\esssup_{U \cap M} |u - f(x)| < \varepsilon.$$
If $f$ is the $L^\infty$ trace of $u$, then $f = u \restr \partial U$ in the sense of Theorem \ref{thm: trace}.

\begin{theorem}[inverse trace theorem, {\cite[Lemma 4.2]{Górny2024}}]\label{thm: Gagliardo}
For every $f \in L^1(\partial M)$ there exists $f^\sharp \in W^{1, 1}(M)$ such that:
\begin{enumerate}
\item $f^\sharp \restr \partial M = f$.
\item $\|f^\sharp\|_{W^{1, 1}(M)} \lesssim \|f\|_{L^1(\partial M)}$.
\item If $f$ is continuous $\mathcal H^{d - 1}$-almost everywhere, then the $L^\infty$ trace of $f^\sharp$ is $f$.
\end{enumerate}
\end{theorem}

%%%%%%%%%%%%%%%%%%%%%%%%%%%%%%%%%
\section{Proof of the max flow/min cut theorem}\label{sec: MFMC proof}
There need not be a area-minimizing real chain in an irrational homology class.
Instead, we shall have to work with primitives of real mass-minimizing currents, which are known as functions of least gradient \cite{gorny2024leastgradient}.
Since we are not assuming $H_{d - 1}((M, \partial M), \RR) = 0$, these primitives only exist on the universal cover $\tilde M$ \cite{daskalopoulos2020transverse}.

To make this precise, we shall use the natural isomorphisms $H^1(M, \RR) = \Hom(\pi_1(M), \RR)$ and $H^1(\partial M, \RR) = \Hom(\pi_1(\partial M), \RR)$.
Given $\hat \alpha \in H^1(M, \RR)$, a function $f$ on the universal cover $\tilde M$ is \dfn{$\hat \alpha$-equivariant} if for every deck transformation $\gamma \in \pi_1(M)$ and $x \in \tilde M$,
$$f(\gamma x) = f(x) + \hat \alpha(\gamma).$$
In other words, $f$ is $\hat \alpha$-equivariant iff $\dif f$ descends to a $1$-form on $M$ whose cohomology class is $\hat \alpha$.

\begin{definition}
Let $\hat \alpha \in H^1(M, \RR)$ and $u \in BV_\loc(\tilde M)$ be $\hat \alpha$-equivariant.
Then $u$ has \dfn{least gradient}, if for every $\hat \alpha$-equivariant $v \in BV_\loc(\tilde M)$ such that $(u - v) \restr \partial \tilde M = 0$,
$$\int_M |\dif u| \dif V \leq \int_M |\dif v| \dif V.$$
A measurable set $E \subseteq M$ is \dfn{perimeter-minimizing} if $1_E$ lifts to a function of least gradient.
\end{definition}

We view $\overline M$ as a codimension-$0$ subset of a larger manifold $N$, by gluing $\overline M$ along $\partial M$ to $\partial M \times \RR_+$ to obtain a smooth manifold $N$ such that $\overline M$ is a strong deformation retract of $N$.
We extend the Riemannian metric to a complete Riemannian metric on $N$.

\begin{definition}\label{dfn: mean curvature barrier condition}
The \dfn{mean curvature barrier condition} asserts that for every sufficiently small simply connected open set $U \subseteq N$ such that $U \cap \partial M \neq \emptyset$, and every measurable set $E \subseteq M$ such that $E \cap U$ is perimeter-minimizing in $U \cap M$, $\partial E \cap \partial M \cap U = \emptyset$.
\end{definition}

The mean curvature barrier condition was introduced in \cite{Jerrard18} as a natural convexity hypothesis to impose when studying the solvability of the Dirichlet problem for functions of least gradient.
The boundary $\partial E$ of the measurable set $E$ is understood in the measure-theoretic sense of \cite[Remark 3.2]{Giusti77}.
The following theorem interprets the mean curvature barrier condition as a weak form of strict mean convexity.
The backwards direction of this theorem, which is the direction we need, is essentially \cite[Theorem 3.1]{Jerrard18}; we include the proof just for the reader's convenience and to emphasize that the forwards direction can be treated by essentially the same techniques.

\begin{theorem}\label{thm: mean curvature barrier condition interpretation}
Let $H_M$ be the mean curvature of $\partial M$.
The mean curvature barrier condition holds iff $\{H_M > 0\}$ is a dense subset of $\partial M$.
\end{theorem}
\begin{proof}
Given $x \in \partial M$, consider normal coordinates $(y, z)$ at $x$, where $y$ is tangent to $\partial M$, $z$ is normal, and $\partial_z$ points into $M$.
Let $w$ be the $C^\infty$ function on $T_x \partial M$, defined near $0$, whose graph is $\partial M$, so $w$ has a double zero at $0$.
Given a $C^2$ function $v$ on $T_x \partial M$, let $Pv(y)$ be the inwards-pointing mean curvature of the graph of $v$ at $(y, v(y))$.
The PDE $Pv = 0$ satisfies the strong maximum principle for $C^2$ solutions, and by calculations as in \cite[\S6.1]{colding2011course}, one easily checks that the linearization of $P$ at $0$ on a ball $B := \{|y| < \varepsilon\}$ is
\begin{equation}\label{eqn: linearized minimal surface equation}
\Delta v + a \cdot \nabla^2 v + Qv = 0
\end{equation}
where $a$ is a $C^\infty$ matrix-valued function such that $\|a\|_{C^0} \lesssim \varepsilon^2$ and $Q$ is a first order linear differential operator whose coefficients have $C^0$ norm $\lesssim \varepsilon$.

We claim that if $\varepsilon$ is small enough, then for every $f \in C^0(\partial B)$ such that $\|f\|_{C^0} \lesssim \varepsilon^2$, there exists $v \in C^2(B)$ such that $Pv = 0$ and $v \restr \partial B = f$.
This is a standard argument which we just sketch.
Letting $C^\beta_w(B)$ denote suitably weighted H\"older spaces, it follows from the interior Schauder estimate \cite[Theorem 6.2]{gilbarg2015elliptic} that if $\varepsilon$ is small enough, then the linearization at $0$ of the smooth mapping
\begin{align*}
C^{2 + \alpha}_w(B) &\to C^\alpha_w(B) \oplus C^0(\partial B) \\
v &\mapsto (Pv, v \restr \partial B)
\end{align*}
is an isomorphism, where $\alpha > 0$ is suitably small.
So the claim holds by the inverse function theorem.

If $\{H_M > 0\}$ is not dense, then there is a nonempty relatively open set $U \subseteq \partial M$ on which $H_M \leq 0$.
If $H_M = 0$ on a relatively open subset of $U$, then on a possibly smaller open set, $M$ is perimeter-minimizing, so the mean curvature barrier condition fails.
Otherwise, we may assume that $H_M < 0$ on $U$ and work in normal coordinates $(y, z)$ based at some $x \in U$.
Using (\ref{eqn: linearized minimal surface equation}) we see that $\Delta w > 0$ near $0$, so by the maximum principle, $w < 0$ on a punctured ball $\{0 < |y| < \varepsilon\}$.
If $\varepsilon > 0$ is chosen small enough, there exists $v \in C^2(B)$ such that $Pv = 0$ on $B$ and $v \restr \partial B = 0$.
By the maximum principle, then, $v \geq w$ on $B$. 
Since $w$ has a double zero at $0$, (\ref{eqn: linearized minimal surface equation}) gives that if $\varepsilon$ was chosen small enough, then $P(-w) > 0$ on $B$, and again by the maximum principle, $v \leq -w$ on $B$.
Therefore the graph $\Sigma$ of $v$ is a minimal hypersurface in $N$ such that $x \in \Sigma$ and $\Sigma \setminus \{x\} \subset M$.
Minimal hypersurfaces are locally area-minimizing, so $\{z > v(y)\}$ is a counterexample to the mean curvature barrier condition.

Conversely, if the mean curvature barrier condition fails, there exists an open set $U \subseteq \overline M$, a perimeter-minimizing set $V \subseteq U \cap M$, and $x \in \partial V \cap \partial M$.
Again we work in normal coordinates based at $x$; if $\varepsilon$ is chosen small enough, then since $w$ has a double zero at $0$, there exists $v \in C^2(B)$ such that $Pv = 0$ on $B$ and $v = w$ on $\partial B$.
If $\partial V$ is the graph of $v$, then $v(0) = 0$, so by the maximum principle, either $v = w$ or $Pw(0) < 0$; either shows that $H_M \leq 0$ near $x$.
Otherwise, $\partial V$ is not the graph of $v$, so by the maximum principle for area-minimizing currents, \cite[Corollary 1]{Simon87}, the graph of $v$ does not intersect $\partial V$.
It follows that $v(0) < 0$, so by the maximum principle, there exists $y_* \in B$ such that $Pw(y_*) < 0$, so $\{H_M > 0\}$ cannot be dense.
\end{proof}

We now establish solvability of the Dirichlet problem for functions of least gradient.
The basic argument is due to Jerrard, Moradifam, and Nachman \cite{Jerrard18}, and was extended by G\'orny for discontinuous boundary data \cite{Górny2024}.
G\'orny's proof does not quite work when $\hat \alpha \neq 0$, because then there is not a good notion of relaxed total variation, but it turns out that the ideas of \cite{Jerrard18} allow us to patch this issue.

\begin{theorem}\label{thm: existence}
Assume the mean curvature barrier condition.
Let $\hat \alpha \in H^1(M, \RR)$, and let $f \in L^1_\loc(\partial \tilde M)$ be $\iota^* \hat \alpha$-equivariant.
If $f$ is continuous $\mathcal H^{d - 1}$-almost everywhere, then there is an $\hat \alpha$-equivariant function $u$ of least gradient such that the $L^\infty$ trace of $u$ is $f$.
\end{theorem}
\begin{proof}
Since $\overline M$ is a strong deformation retract of $N$, there is a canonical isomorphism $H^1(M, \RR) = H^1(N, \RR)$.
Let $h$ be an $\hat \alpha$-equivariant Lipschitz function on $\tilde N$.
By Theorem \ref{thm: Gagliardo} (applied to $M \setminus N$ rather than $M$), we can extend $f - h$ to a function $(f^\sharp - h) \in W^{1, 1}(N \setminus M)$ such that the $L^\infty$ trace of $f^\sharp$ is $f$.

Let us say that a function $w \in BV_\loc(\tilde N)$ is \dfn{admissible} if $w$ is $\hat \alpha$-equivariant and $w = f^\sharp$ on $\tilde N \setminus \tilde M$.
Let $(u_n)$ be a sequence of admissible functions such that $\int_M |\dif u_n| \dif V$ is approaching the infimum of $\int_M |\dif w| \dif V$ among all admissible functions.
The set of admissible functions is easily seen to be closed for the weakstar topology on $BV_\loc$, so after taking a subsequence we may assume that $u_n \weakto^* u$ for some admissible $u$.

Suppose towards contradiction that the $L^\infty$ trace of $u$ is not $f$.
By the contradiction assumption, up to a sign, there exist $x \in \partial \tilde M$ and $\delta > 0$ such that for every neighborhood $U$ of $x$,
$$\esssup_U u \geq f(x) + \delta.$$
Let $U$ be a sufficiently small neighborhood of $x$, so that $U$ witnesses the mean curvature barrier condition.
Since $u$ is admissible,
$$\{u > f(x) + \delta/2\} \setminus \tilde M = \emptyset.$$
By \cite[Theorem 2.4]{Jerrard18}, $\{u > f(x) + \delta/2\}$ is perimeter-minimizing, so by \cite[Lemma 4.2]{Górny2024} and the mean curvature barrier condition, $x$ is an interior point of $\tilde N \setminus \tilde M$, which is a contradiction.

If $v \in BV_\loc(\tilde M)$ is $\hat \alpha$-equivariant and $(u - v) \restr \partial \tilde M = 0$, then we can extend $v$ to an admissible function by setting $v \restr (\tilde N \setminus \tilde N) := f^\sharp$.
So by definition of $u$, $\int_M |\dif u| \dif V \leq \int_M |\dif v| \dif V$.
Therefore $u$ has least gradient.
\end{proof}

Next, we construct the dual vector field (or dual calibration) to a function of least gradient.
This is a special case of \cite[Theorem 4.10(1)]{Federer1974}, which itself is an easy corollary of the Hanh-Banach theorem.
We give a separate proof based on the $p$-Laplacian, which generalizes \cite{Mazon14,daskalopoulos2020transverse}.

\begin{lemma}\label{lma: Hanh Banach}
Let $\hat \alpha \in H^1(M, \RR)$, and let $f \in L^1_\loc(\partial \tilde M)$ be $\iota^* \hat \alpha$-equivariant.
Let $\mathcal C$ be the set of all $\hat \alpha$-equivariant functions $u \in BV_\loc(\tilde M)$ such that $u \restr \partial \tilde M = f$.
Then there exists $\gamma \in L^\infty(M, \Omega^{d - 1}_\cl)$ such that $\|\gamma\|_{L^\infty} \leq 1$ and for each $u \in \mathcal C$,
\begin{equation}\label{eqn: calibration}
\inf_{u^* \in \mathcal C} \int_M |\dif u^*| \dif V = \int_M \dif u \wedge \gamma.
\end{equation}
\end{lemma}
\begin{proof}
Let us normalize $\vol(M) = 1$, and let $h$ be an $\hat \alpha$-equivariant function on $\tilde M$ which is Lipschitz up to the boundary.

We first assume that $f$ is Lipschitz, which case $f$ has a $\hat \alpha$-equivariant extension to $\overline M$ which is Lipschitz up to the boundary, and in particular is in $W^{1, p}$.
So for each $p > 1$, there exists an $\hat \alpha$-equivariant solution $u_p$ of the $p$-Laplacian, 
\begin{equation}\label{eqn: pLaplacian}
\dif(|\dif u_p|^{p - 2} \star \dif u_p) = 0,
\end{equation}
such that $u_p \restr \partial \tilde M = f$, and moreover $u_p$ minimizes $J_p(v) := \int_M |\dif v|^p \dif V$ subject to those constraints.
This is because $J_p$ is strictly convex and coercive on $W^{1, p}$ and its Euler-Lagrange equation is (\ref{eqn: pLaplacian}).\footnote{See \cite[Chapter 2]{lindqvist2019notes} for the usual reference on the $p$-Laplacian, or \cite[\S2.1]{daskalopoulos2020transverse} for more details in the equivariant setting.}

Let $p'$ be the solution of $1/p + 1/p' = 1$.
Following \cite[\S3.1]{daskalopoulos2020transverse}, let
$$\gamma_p := |\dif u_p|^{p - 2} \star \dif u_p$$
be the \dfn{conjugate $p'$-harmonic form} of $u_p$; by (\ref{eqn: pLaplacian}), $\dif \gamma_p = 0$.
Since $u_p$ minimizes $J_p$ and $p \leq 2$, we obtain by H\"older's inequality,
$$\int_M |\gamma_p|^{p'} \dif V = \int_M |\dif u_p|^p \dif V \leq \int_M |\dif u_2|^p \dif V \leq \left(\int_M |\dif u_2|^2 \dif V\right)^{p/2}.$$
Therefore for each $q \geq 2$, H\"older's inequality gives
$$\limsup_{p \to 1} \|\gamma_p\|_{L^q} \leq \limsup_{p \to 1} \|\gamma_p\|_{L^{p'}} \leq \limsup_{p \to 1} \left(\int_M |\dif u_2|^2 \dif V\right)^{\frac{p}{2p'}} \leq 1.$$
So by Alaoglu's theorem, there exists $\gamma$ such that $\gamma_p \weakto \gamma$ in $L^2(M, \Omega^{d - 1}_\cl)$ along a subsequence as $p \to 1$ and $\|\gamma\|_{L^\infty} \leq 1$.
Then
$$\inf_{u^* \in \mathcal C} \int_M |\dif u^*| \dif V \leq \liminf_{p \to 1} \int_M |\dif u_p| \dif V \leq \liminf_{p \to 1} \int_M |\dif u_p|^p \dif V = \liminf_{p \to 1} \int_M \dif u_p \wedge \gamma_p.$$
By the normal trace theorem \cite[\S2.4]{cessenat1996mathematical}, $\iota^* \gamma_p \weakto \iota^* \gamma$ in $W^{-1/2, 2}(\partial M, \Omega^{d - 1})$, the dual space of the fractional Sobolev space $W^{1/2, 2}(\partial M)$.
But $f - h$ is Lipschitz, and in particular is in $W^{1/2, 2}(\partial M)$, so
\begin{align*}
\liminf_{p \to 1} \int_M \dif u_p \wedge \gamma_p 
&= \liminf_{p \to 1} \int_{\partial M} (f - h) \gamma_p + \int_M \dif h \wedge \gamma_p \\
&= \int_{\partial M} (f - h) \gamma + \int_M \dif h \wedge \gamma \\
&= \int_M \dif u \wedge \gamma
\end{align*}
which proves one direction of (\ref{eqn: calibration}).
The other direction follows because $\|\gamma\|_{L^\infty} \leq 1$.

In the general case, choose $\iota^* \hat \alpha$-equivariant Lipschitz functions $f_n$ such that $f_n \to f$ in $L^1_\loc(\partial \tilde M)$.
Let $\mathcal C_n$ be the set of all $\hat \alpha$-equivariant functions $u_n \in BV_\loc(\tilde M)$ such that $u_n \restr \partial \tilde M = f_n$.
Since $f_n \in W^{1/2, 2}(\partial \tilde M)$, the above argument shows that there exists $\gamma_n \in L^\infty(M, \Omega^{d - 1}_\cl)$ such that $\|\gamma_n\|_{L^\infty} \leq 1$ and
$$\inf_{u_n^* \in \mathcal C_n} \int_M |\dif u_n^*| \dif V = \int_{\partial M} (f_n - h) \gamma_n + \int_M \dif h \wedge \gamma_n.$$
By Alaoglu's theorem, there exists $\gamma$ such that after taking a subsequence, $\gamma_n \weakto^* \gamma$ in $L^\infty(M, \Omega^{d - 1}_\cl)$.
In particular, $\|\gamma\|_{L^\infty} \leq 1$ and $\int_M \dif h \wedge \gamma_n \to \int_M \dif h \wedge \gamma$.
By Theorem \ref{thm: Gagliardo}, there exist $w_n, w \in W^{1, 1}(M)$, which extend $f_n - h, f - h$, such that $w_n \to w$ in $W^{1, 1}(M)$.
Therefore 
$$\int_{\partial M} (f - h) \gamma = \int_M \dif w \wedge \gamma = \lim_{n \to \infty} \int_M \dif w_n \wedge \gamma_n = \lim_{n \to \infty} \int_{\partial M} (f_n - h) \wedge \gamma_n.$$
Choose $u_n \in \mathcal C_n$ such that 
$$\int_M |\dif u_n| \dif V \leq \frac{1}{n} + \inf_{u_n^* \in \mathcal C_n} \int_M |\dif u_n^*| \dif V.$$
Then, since $\|\gamma\|_{L^\infty} \leq 1$,
$$\liminf_{n \to \infty} \int_M |\dif u_n| \dif V \leq \int_{\partial M} (f - h) \gamma + \int_M \dif h \wedge \gamma = \int_M \dif u \wedge \gamma \leq \inf_{u^* \in \mathcal C} \int_M |\dif u^*| \dif V.$$
By Theorem \ref{thm: Gagliardo}, there exists $v_n \in W^{1, 1}(M)$ such that $v_n \restr \partial \tilde M = f_n - f$ and
$$\int_M |\dif v_n| \dif V \lesssim \int_{\partial M} |f_n - f| \dif \mathcal H^{d - 1},$$
so, since $f_n - f \to 0$ in $L^1(\partial M)$,
\begin{align*}
\inf_{u^* \in \mathcal C} \int_M |\dif u^*| \dif V 
&\leq \liminf_{n \to \infty} \int_M |\dif u_n - \dif v_n| \dif V 
= \liminf_{n \to \infty} \int_M |\dif u_n| \dif V.
\end{align*}
The above inequalities collapse and prove (\ref{eqn: calibration}).
\end{proof}

\begin{proof}[Proof of Theorem \ref{thm: MFMC}]
Let $\beta := -\partial \alpha$.
Since $\partial S = 0$, there exists a function $f$ on $\partial \tilde M$ such that for every $\varphi \in C^\infty(\partial M, \Omega^{d - 2})$,
\begin{equation}\label{eqn: boundary condition satisfied}
\int_S \varphi = -\int_{\partial M} \dif f \wedge \varphi.
\end{equation}
Since $S$ has finite mass, $f \in BV_\loc(\tilde M)$, so $f$ is continuous $\mathcal H^{d - 1}$-almost everywhere \cite[Theorem 3.78]{Ambrosio2000}.
The definition (\ref{eqn: Lefschetz isomorphism}) of Lefschetz duality, and the fact that $[S] = [\partial C] = \partial \alpha$, implies that $f$ is $\hat \beta$-equivariant.
By Lemma \ref{lma: Lefschetz duality commutes with pullback}, $\iota^* \hat \alpha = -\hat \beta$, so by Theorems \ref{thm: mean curvature barrier condition interpretation} and \ref{thm: existence}, there exists an $\hat \alpha$-equivariant function $u$ of least gradient such that $u \restr \partial \tilde M = f$.

Let $\flux F^*$ be the $d - 1$-form given by Lemma \ref{lma: Hanh Banach}, let $C^*$ be the current such that for every $\psi \in C^\infty_\cpt(M, \Omega^{d - 1})$,
$$\int_{C^*} \psi = \int_M \dif u \wedge \psi.$$
Therefore by (\ref{eqn: boundary condition satisfied}) we have
\begin{equation}\label{eqn: mass is given by flux}
\mass(C^*) = \int_M |\dif u| \dif V = \int_M \langle \dif u, F^*\rangle \dif V = \int_{C^*} \flux F^*
\end{equation}
by (\ref{eqn: boundary current is trace}) we have for every $\varphi \in C^1(\partial M, \Omega^{d - 2})$,
$$\int_{\partial C^*} \varphi = \int_{\partial M} f \dif \varphi = -\int_{\partial M} \dif f \wedge \varphi = \int_S \varphi,$$
so $\partial C^* = S$, and $[C^*] = \alpha$.

If $C$ is another current such that $\partial C = S$ and $[C] = \alpha$, then by (\ref{eqn: mass is given by flux}),
$$\mass(C^*) = \int_{C^*} \flux F^* = \int_C \flux F^* \leq \mass(C).$$
Therefore $C^*$ minimizes its mass. 
Dually, for every $F \in \mathcal F$, we obtain by (\ref{eqn: mass is given by flux}) that
$$\int_{C^*} \flux F \leq \mass(C^*) = \int_{C^*} \flux F^*,$$
so $F^*$ realizes the maximum in (\ref{eqn: MFMC}).
\end{proof}

%%%%%%%%%%%%%%%%
\section{Epilogue}\label{sec: epilogue}
\subsection{Anisotropic max flow/min cut theorem}
In many applications of functions of least gradient, the background geometry is not Riemannian, but is instead induced by an elliptic integrand.
To state the definition, let $\Omega^k_x$ denote the fiber of $\Omega^k$ at $x \in M$; if $\omega, \overline \omega \in \Omega^d_x$, $\omega \geq \overline \omega$ means that $\omega - \overline \omega$ is positively oriented.

\begin{definition}\label{dfn: elliptic integrand}
A \dfn{elliptic integrand} on $\overline M$ is a morphism of vector bundles $\phi: \Omega^1 \to \Omega^d$ such that for each $x \in N$, $\alpha, \beta \in \Omega^1_x$, and $t \in \RR$:
\begin{enumerate}
\item $\phi(\alpha) \geq 0$, and if $\phi(\alpha) = 0$, then $\alpha = 0$.
\item $\phi(\alpha) + \phi(\beta) \leq \phi(\alpha) + \phi(\beta)$.
\item $\phi(t\alpha) = |t| \phi(\alpha)$.
\end{enumerate}
\end{definition}

Thus an elliptic integrand induces an $L^1$ norm on measurable $1$-forms, so the notions of function of least gradient, mass-minimizing current, and mean curvature barrier condition still make sense.
By duality, an elliptic integrand induces a notion of $L^\infty$ norm on measurable $d - 1$-forms, so the definition of $\mathcal F$ also still makes sense.
Every theorem we have cited about functions of least gradient and the mean curvature barrier condition were proved only under the assumption that the background geometry is a $C^0$ elliptic integrand.
The theory of currents of finite mass, the proof of Theorem \ref{thm: existence}, and the proof of Lemma \ref{lma: Hanh Banach} which uses the Hanh-Banach theorem, all go through with minor changes.
We only used the strict mean convexity of $\partial M$ to prove the mean curvature barrier condition, and only used the fact that $S$ has finite mass to show that it has a primitive which is continuous $\mathcal H^{d - 1}$-almost everywhere.
We thus obtain a stronger version of Theorem \ref{thm: MFMC}:

\begin{theorem}
Let $\overline M = M \cup \partial M$ be a compact oriented manifold-with-boundary equipped with a $C^0$ elliptic integrand which satisfies the mean curvature barrier condition.
Let $\alpha \in H_{d - 1}((M, \partial M), \RR)$ and let $S$ be a closed $d - 2$-current in $\partial M$ such that $[S] = \partial \alpha$.

If $S$ admits a primitive in $L^1_\loc(\partial M)$ which is continuous $\mathcal H^{d - 1}$-almost everywhere, then there exists a $d - 1$-current $C^*$ which minimizes its mass among all currents $C$ such that $\partial C = S$ and $[C] = \alpha$, and
$$\mass(C^*) = \max_{F^* \in \mathcal F} \int_{C^*} \flux F^*.$$
\end{theorem}

\subsection{The Freedman--Headrick discretization}\label{sec: conjectures}
Freedman and Headrick, \cite[Appendix]{Freedman_2016}, have proposed to discretize $M$ by an undirected graph $(V, E)$.
Fix a pair of sufficiently small scales $\varepsilon_1 \ll \varepsilon_0$, let $V$ be the set of centers of $\varepsilon_1$-balls which form a maximal packing of $M$, and let $(x, y) \in E$ iff $||x - y| - \varepsilon_0| \ll \varepsilon_0$; thus for each $x \in V$, $\{y \in V: (x, y) \in E\}$ discretizes $\partial B(x, \varepsilon_0)$.
They conjectured that the maximal flow and minimal cut on $(V, E)$ converge as $\max(\varepsilon_0, \varepsilon_1/\varepsilon_0) \to 0$ to the continuum maximum flow and minimal cut.
I do not know how to prove this conjecture, but highlight it because it seems quite interesting, and would give a proof of Theorem \ref{thm: MFMC} which is constructive, in the sense that one could use the Fulkerson-Ford algorithm to construct area-minimizing hypersurfaces and their dual calibrations.

Classically, the discrete max flow/min cut theorem does not involve any sort of homology.
But suppose $(V, E)$ was obtained by discretizing $M$; by taking the persistent homology of $(V, E)$ at a sufficiently small scale $\gg \varepsilon_0$, presumably one recovers the homology of $M$.
I conjecture, therefore, that there is a version of the discrete max flow/min cut theorem in which the cut must be in a certain persistent homology class.

\subsection{Thurston's max flow/min cut theorem}
The following theorem grew out of Thurston's study of his asymmetric metric on Teichm\"uller space.

\begin{theorem}[{\cite{Thurston98,Gu_ritaud_2017}}]\label{thm: Thurston}
Let $M$ be a closed hyperbolic surface, and let $M'$ be a closed hyperbolic surface or $\Sph^1$.
Let $\rho$ be a homotopy class of maps $M \to M'$ and let $L$ be the least Lipschitz constant of a map in $\rho$.
If $L > 1$ or $M' = \Sph^1$, then there exists a measured geodesic lamination $\lambda$ in $M$ such that for every minimizing Lipschitz map $f$ in $\rho$, the pushforward current $f_* \lambda$ satisfies
$$\mass(f_* \lambda) = L \cdot \mass(\lambda).$$
\end{theorem}

Thurston conjectured that Theorem \ref{thm: Thurston} should have ``a simpler proof based on more general principles -- in particular, the max flow min cut principle, convexity, and $L^0 \leftrightarrow L^\infty$ duality'' which ``fits into a context including $L^p$ comparisons'' \cite[Abstract]{Thurston98}.
Daskalopoulos and Uhlenbeck, \cite{daskalopoulos2020transverse,daskalopoulos2022,daskalopoulos2023}, have recently given a proof of Theorem \ref{thm: Thurston} which arguably meets all of Thurston's criteria, except possibly for the use of the max flow/min cut theorem.
But their proof directly inspired our proof of Lemma \ref{lma: Hanh Banach}, so arguably we see that their proof satisfies \emph{all} of Thurston's criteria.
In fact, the $M' = \Sph^1$ case of Theorem \ref{thm: Thurston} is a special case of Corollary \ref{crly: lamination MFMC}:

\begin{proof}[Proof when $M' = \Sph^1$]
The \dfn{stable norm} $\|\alpha\|_1$ of a homology class $\alpha \in H_1(M, \RR)$ is the infimum of $\len(\gamma)$, among all loops $\gamma$ in $\alpha$.
We identify a homotopy class of maps $M \to \Sph^1$ with its pushforward $H_1(M, \RR) \to H_1(\Sph^1, \RR) = \RR$; thus we have an isomorphism $[M, \Sph^1] \cong H^1(M, \RR)$.
By Lemma \ref{lma: Hanh Banach}, the dual norm $\|\rho\|_\infty$ of the stable norm is the infimum of $\|\psi\|_{L^\infty}$ among all $1$-forms $\psi$ in $\rho \in H^1(M, \RR)$.
If we normalize $\|\rho\|_\infty = 1$, then we can find $\alpha$ such that
\begin{equation}\label{eqn: stable duality}
\|\alpha\|_1 = \langle \rho, \alpha\rangle = 1.
\end{equation}
Then $\lambda$ is the lamination given by Corollary \ref{crly: lamination MFMC}, and for every minimizing Lipschitz map $f$ in the homotopy class $\rho$, (\ref{eqn: stable duality}) implies that $F^* := \nabla^\perp f$ satisfies (\ref{eqn: MFMC}).
\end{proof}

In \cite{BackusBest2} we study the stable norm extensively.
The $M' = \Sph^1$ case of Theorem \ref{thm: Thurston} is also a special case of the main theorem of that paper.

%%%%%%%%%%%%%%%%%%%%%%%

\printbibliography

@article{BackusCML,
author={Backus, Aidan},
title={Minimal Laminations and Level Sets of 1-Harmonic Functions},
journal={The Journal of Geometric Analysis},
year={2024},
month={8},
day={08},
volume={34},
number={10},
pages={309},
issn={1559-002X},
doi={10.1007/s12220-024-01758-8},
url={https://doi.org/10.1007/s12220-024-01758-8}
}

@misc{BackusBest2,
      title={The canonical lamination calibrated by a cohomology class}, 
      author={Aidan Backus},
      year={2024},
      eprint={2412.00255},
      archivePrefix={arXiv},
      primaryClass={math.DG},
      url={https://arxiv.org/abs/2412.00255}, 
}

@article{Strang1983,
author={Strang, Gilbert},
title={Maximal flow through a domain},
journal={Mathematical Programming},
year={1983},
month={06},
day={01},
volume={26},
number={2},
pages={123-143},
issn={1436-4646},
doi={10.1007/BF02592050},
url={https://doi.org/10.1007/BF02592050}
}

@article{Nozawa90,
author = {Ryōhei Nozawa},
title = {{Max-flow min-cut theorem in an anisotropic network}},
volume = {27},
journal = {Osaka Journal of Mathematics},
number = {4},
publisher = {Osaka University and Osaka Metropolitan University, Departments of Mathematics},
pages = {805 -- 842},
year = {1990},
}

@article{Gu_ritaud_2017,
	doi = {10.2140/gt.2017.21.693},
  
	url = {https://doi.org/10.2140%2Fgt.2017.21.693},
  
	year = 2017,
  
	publisher = {Mathematical Sciences Publishers},
  
	volume = {21},
  
	number = {2},
  
	pages = {693--840},
  
	author = {Fran{\c{c}
}ois Gu{\'{e}}ritaud and Fanny Kassel},
  
	title = {Maximally stretched laminations on geometrically finite hyperbolic manifolds},
  
	journal = {Geometry and Topology}
}

@article{Grieser05,
author = {Grieser, Daniel},
year = {2005},
month = {07},
pages = {},
title = {The first eigenvalue of the Laplacian, isoperimetric constants, and the Max Flow Min Cut Theorem},
volume = {87},
journal = {Archiv der Mathematik},
doi = {10.1007/s00013-005-1623-4}
}

@book{gilbarg2015elliptic,
  title={Elliptic Partial Differential Equations of Second Order},
  author={Gilbarg, D. and Trudinger, N.S.},
  isbn={9783642617980},
  series={Classics in Mathematics},
  url={https://books.google.com/books?id=l9L6CAAAQBAJ},
  year={2015},
  publisher={Springer Berlin Heidelberg}
}

@book{gorny2024leastgradient,
  title={Functions of Least Gradient},
  author={Wojciech Górny and José M. Mazón},
  isbn={978-3-031-51880-5},
  series={Monographs in Mathematics},
  year={2024},
  publisher={Birkhäuser Cham}
}

@article{Anzellotti1983,
author={Anzellotti, Gabriele},
title={Pairings between measures and bounded functions and compensated compactness},
journal={Annali di Matematica Pura ed Applicata},
year={1983},
month={12},
day={01},
volume={135},
number={1},
pages={293-318},
issn={1618-1891},
doi={10.1007/BF01781073},
url={https://doi.org/10.1007/BF01781073}
}

@book{simon1983GMT,
  title={Lectures on Geometric Measure Theory},
  author={Simon, L.},
  isbn={9780867844290},
  year={1983},
  series={Proceedings of the Centre for Mathematical Analysis},
  publisher={Centre for Mathematical Analysis, Australian National University}
}

@book{papadimitriou1982combinatorial,
  title={Combinatorial Optimization: Algorithms and Complexity},
  author={Papadimitriou, C.H. and Steiglitz, K.},
  isbn={9780131524620},
  lccn={81005866},
  url={https://books.google.com/books?id=DvBQAAAAMAAJ},
  year={1982},
  publisher={Prentice Hall}
}

@article{Górny2024,
author={Górny, Wojciech},
title={Least gradient problem with Dirichlet condition imposed on a part of the boundary},
journal={Calculus of Variations and Partial Differential Equations},
year={2024},
month={2},
day={12},
volume={63},
number={2},
pages={58},
issn={1432-0835},
doi={10.1007/s00526-023-02646-9},
}

@article{Spradlin14,
 ISSN = {00222518, 19435258},
 URL = {http://www.jstor.org/stable/24904268},
 author = {Gregory S. Spradlin and Alexandru Tamasan},
 journal = {Indiana University Mathematics Journal},
 number = {6},
 pages = {1819--1837},
 publisher = {Indiana University Mathematics Department},
 title = {Not All Traces on the Circle Come from Functions of Least Gradient in the Disk},
 urldate = {2023-11-17},
 volume = {63},
 year = {2014}
}

@book{cessenat1996mathematical,
  title={Mathematical Methods In Electromagnetism: Linear Theory And Applications},
  author={Cessenat, M.},
  isbn={9789814525381},
  series={Series On Advances In Mathematics For Applied Sciences},
  url={https://books.google.com/books?id=kZTsCgAAQBAJ},
  year={1996},
  publisher={World Scientific Publishing Company}
}

@book{colding2011course,
  title={A Course in Minimal Surfaces},
  author={Colding, T.H. and Minicozzi, W.P.},
  isbn={9780821853238},
  lccn={2010044373},
  series={Graduate studies in mathematics},
  url={https://books.google.com/books?id=k4ODAwAAQBAJ},
  year={2011},
  publisher={American Mathematical Society}
}

@article{Mazon14,
 ISSN = {00222518, 19435258},
 URL = {http://www.jstor.org/stable/24904253},
 author = {José M. Mazón and Julio D. Rossi and Sergio Segura de León},
 journal = {Indiana University Mathematics Journal},
 number = {4},
 pages = {1067--1084},
 publisher = {Indiana University Mathematics Department},
 title = {Functions of Least Gradient and 1-Harmonic Functions},
 urldate = {2023-01-25},
 volume = {63},
 year = {2014}
}

@misc{Thurston98,
  doi = {10.48550/ARXIV.MATH/9801039},
  
  url = {https://arxiv.org/abs/math/9801039},
  
  author = {Thurston, William P.},
  
  title = {Minimal stretch maps between hyperbolic surfaces},
  
  publisher = {arXiv},
  
  year = {1998},
  
  copyright = {Assumed arXiv.org perpetual, non-exclusive license to distribute this article for submissions made before January 2004}
}

@article{daskalopoulos2020transverse,
author = {Georgios Daskalopoulos and Karen Uhlenbeck},
title = {{Transverse measures and best Lipschitz and least gradient maps}},
volume = {127},
journal = {Journal of Differential Geometry},
number = {3},
publisher = {Lehigh University},
pages = {969 -- 1018},
year = {2024},
doi = {10.4310/jdg/1721071495},
URL = {https://doi.org/10.4310/jdg/1721071495}
}

@misc{daskalopoulos2022,
  doi = {10.48550/ARXIV.2205.08250},
  
  url = {https://arxiv.org/abs/2205.08250},
  
  author = {Daskalopoulos, Georgios and Uhlenbeck, Karen},
  
  title = {Analytic properties of Stretch maps and geodesic laminations},
  
  publisher = {arXiv},
  
  year = {2022},
  month={5},
  day={17},
  
  copyright = {arXiv.org perpetual, non-exclusive license}
}

@misc{daskalopoulos2023,
      title={Best Lipschitz maps and Earthquakes}, 
      author={Georgios Daskalopoulos and Karen Uhlenbeck},
      year={2024},
      eprint={2410.08296},
      archivePrefix={arXiv},
      primaryClass={math.DG},
      url={https://arxiv.org/abs/2410.08296}, 
}

@book{lindqvist2019notes,
  title={Notes on the Stationary p-Laplace Equation},
  author={Lindqvist, P.},
  isbn={9783030145019},
  series={SpringerBriefs in Mathematics},
  url={https://books.google.com/books?id=YWqVDwAAQBAJ},
  year={2019},
  publisher={Springer International Publishing}
}

@article{Harvey82,
author = {Reese Harvey and H. Blaine Lawson},
title = {{Calibrated geometries}},
volume = {148},
journal = {Acta Mathematica},
number = {none},
publisher = {Institut Mittag-Leffler},
pages = {47 -- 157},
year = {1982},
doi = {10.1007/BF02392726},
URL = {https://doi.org/10.1007/BF02392726}
}

@book{Giusti77,
  title={Minimal Surfaces and Functions of Bounded Variation},
  author={E. Giusti},
  isbn={9780817631536},
  lccn={lc84011195},
  series={Monographs in Mathematics},
  url={https://books.google.com/books?id=dNgsmArDoeQC},
  year={1984},
  publisher={Birkh{\"a}user Boston}
}

@article{Freedman_2016,
	doi = {10.1007/s00220-016-2796-3},
  
	url = {https://doi.org/10.1007%2Fs00220-016-2796-3},
  
	year = 2016,
	month = 11,
  
	publisher = {Springer Science and Business Media {LLC}
},
  
	volume = {352},
  
	number = {1},
  
	pages = {407--438},
  
	author = {Michael Freedman and Matthew Headrick},
  
	title = {Bit Threads and Holographic Entanglement},
  
	journal = {Communications in Mathematical Physics}
}

@book{hörmander2007analysis,
  title={The Analysis of Linear Partial Differential Operators III: Pseudo-Differential Operators},
  author={H{\"o}rmander, L.},
  isbn={9783540499374},
  lccn={89026134},
  series={Classics in Mathematics},
  url={https://books.google.com/books?id=b55IBO0FaOYC},
  year={2007},
  publisher={Springer Berlin Heidelberg}
}

@article{Federer1974,
 ISSN = {00222518, 19435258},
 URL = {http://www.jstor.org/stable/24890827},
 author = {Herbert Federer},
 journal = {Indiana University Mathematics Journal},
 number = {4},
 pages = {351--407},
 publisher = {Indiana University Mathematics Department},
 title = {Real Flat Chains, Cochains and Variational Problems},
 urldate = {2023-05-19},
 volume = {24},
 year = {1974}
}

@book{sullivan1990crystalline,
  title={A Crystalline Approximation Theorem for Hypersurfaces},
  author={Sullivan, J.M.},
  url={https://books.google.com/books?id=I-WuYgEACAAJ},
  year={1990},
  publisher={Princeton University}
}

@article{Simon87,
author = {Leon Simon},
title = {{A strict maximum principle for area minimizing hypersurfaces}},
volume = {26},
journal = {Journal of Differential Geometry},
number = {2},
publisher = {Lehigh University},
pages = {327 -- 335},
year = {1987},
doi = {10.4310/jdg/1214441373},
URL = {https://doi.org/10.4310/jdg/1214441373}
}

@article{Jerrard18,
url = {https://doi.org/10.1515/crelle-2014-0151},
title = {Existence and uniqueness of minimizers of general least gradient problems},
author = {Robert L. Jerrard and Amir Moradifam and Adrian I. Nachman},
pages = {71--97},
volume = {2018},
number = {734},
journal = {Journal für die reine und angewandte Mathematik (Crelles Journal)},
doi = {doi:10.1515/crelle-2014-0151},
year = {2018},
lastchecked = {2024-11-01}
}

@book{Ambrosio2000,
    author = {Ambrosio, Luigi and Fusco, Nicola and Pallara, Diego},
    title = {Functions of Bounded Variation and Free Discontinuity Problems},
    publisher = {Oxford University Press},
    year = {2000},
    month = {03},
    isbn = {9780198502456},
    doi = {10.1093/oso/9780198502456.001.0001},
    url = {https://doi.org/10.1093/oso/9780198502456.001.0001},
}

@article{Dweik2019,
author={Dweik, Samer
and Santambrogio, Filippo},
title={$L^p$ bounds for boundary-to-boundary transport densities, and $W^{1, p}$ bounds for the BV least gradient problem in 2D},
journal={Calculus of Variations and Partial Differential Equations},
year={2019},
month={1},
day={04},
volume={58},
number={1},
pages={31},
issn={1432-0835},
doi={10.1007/s00526-018-1474-z},
url={https://doi.org/10.1007/s00526-018-1474-z}
}

\end{document}